\documentclass[12pt,reqno]{amsart}
\usepackage{fullpage}
\usepackage{times}
\usepackage{amsmath,amssymb,amsthm,url}
\usepackage[utf8]{inputenc}
\usepackage[english]{babel}
\usepackage{comment}
\usepackage{bbm}
\usepackage{enumerate}
\usepackage{bm}
\usepackage{graphicx}
\usepackage{mathrsfs}
\usepackage[colorlinks=true, pdfstartview=FitH, linkcolor=blue, citecolor=blue, urlcolor=blue]{hyperref}

\newtheorem{thm}{Theorem}[section]
\newtheorem{lem}[thm]{Lemma}
\newtheorem{prop}[thm]{Proposition}
\newtheorem{cor}[thm]{Corollary}
\newtheorem{conj}[thm]{Conjecture}

\theoremstyle{definition}

\newtheorem{rem}[thm]{Remark}

\newcommand{\C}{\mathbb{C}}

\newcommand{\F}{\mathbb{F}}

\newcommand{\rad}{\operatorname{rad}}
\renewcommand{\pmod}[1]{{\ifmmode\text{\rm\ (mod~$#1$)}\else\discretionary{}{}{\hbox{ }}\rm(mod~$#1$)\fi}}

\begin{document}

\title{Distribution of power residues over shifted subfields and maximal cliques in generalized Paley graphs}
\author{Greg Martin}
\address{Department of Mathematics \\ University of British Columbia \\ Vancouver  V6T 1Z2 \\ Canada}
\email{gerg@math.ubc.ca}
\author{Chi Hoi Yip}
\address{School of Mathematics\\ Georgia Institute of Technology\\ GA 30332\\ United States}
\email{cyip30@gatech.edu}

\begin{abstract}
We derive an asymptotic formula for the number of solutions in a given subfield to certain system of equations over finite fields. As an application, we construct new families of maximal cliques in generalized Paley graphs. Given integers $d\ge2$ and $q \equiv 1 \pmod d$, we show that for each positive integer $m$ such that $\operatorname{rad}(m) \mid \operatorname{rad}(d)$, there are maximal cliques of size approximately $q/m$ in the $d$-Paley graph defined on $\mathbb{F}_{q^d}$. We also confirm a conjecture of Goryainov, Shalaginov, and Yip on the maximality of certain cliques in generalized Paley graphs, as well as an analogous conjecture of Goryainov for Peisert graphs.
\end{abstract}

\maketitle

\section{Introduction}

Throughout this paper,~$q$ will always denote a prime power and $\F_q$ the finite field with $q$ elements, and we write $\F_q^*=\F_q \setminus \{0\}$. 

Let $d\geq 2$ be an integer, and let $q\equiv 1 \pmod d$ be a sufficiently large prime power. It is well known that the set of $d$-th powers (of nonzero elements) in $\F_q$ behaves like a random set~\cite{S01}. There are different ways to make this statement precise. One typical way to do so is to consider the following well-known lemma \cite[Exercise 5.66]{LN97}, which follows from a standard application of Weil's bound on complete character sums.
\begin{lem}\label{lem1}
Let $v_{1}, \ldots, v_{k}$ be distinct elements of the finite field $\F_q$ and let $d \mid (q-1)$, where $d \geq 2$. Let $M$ be the number of solutions $x \in \F_q$ to the system of equations $\left(x-v_i\right)^{(q-1) / d}=1$ $(i=1, \ldots, k)$. Then $|M-q / d^{k}|\leq k\sqrt{q}$.
\end{lem}
In other words, the number of $x \in \F_q$ such that $x-v_i$ is a $d$-th power for all $i$ is asymptotically equal to $q/d^k$. Note that if we model the set of $d$-th powers as a random subset $S$ of $\F_q$ with density $1/d$, then the expected number of $x \in \F_q$ such that each $x-v_i \in S$ would also be $q/d^k$; in this sense, the set of $d$-th powers in $\F_q$ behaves like a random set. Lemma~\ref{lem1} and its generalizations have plenty of applications in number theory, combinatorics, and finite geometry; we refer the reader to an excellent survey by Sz\H{o}nyi \cite{S97}.

While Lemma~\ref{lem1} states that globally the solutions behave randomly, it is desirable to obtain refined information on the distribution of these solutions for various applications. In this paper, we aim to understand the ``local distribution" of these solutions within a given subfield, or equivalently the distribution of power residues in a shifted subfield. More precisely, given a field extension $K/L$ of finite fields with $|L| \equiv 1 \pmod d$, we want to count the number of solutions of $(x-v_i)^{(|K|-1)/d}=1$ in the base field $L$, where $v_{1}, \ldots, v_{k}\in K$. Without loss of generality, assume that $L=\F_q$ and $K=\F_{q^n}$. 
Our first main result gives an asymptotic formula for the number of solutions.

\begin{thm}\label{thm:main}
Let $d \geq 2$ and $q \equiv 1 \pmod d$.
For $1\leq i \leq k$,
let $v_i \in \F_{q^n}$ be of degree~$d_i$ over~$\F_q$, and
let~$M$ be the number of solutions~$x$ in the base field~$\F_q$ to the system of~$k$ equations $(x-v_i)^{(q^n-1)/d}=1$.
Suppose that $v_i$ and $v_j$ are not Galois conjugates with respect to the field extension $\F_{q^n}/\F_q$ whenever $i \neq j$.
Then
$$
\bigg |M- q \prod_{i=1}^k  \frac{\gcd(dd_i,n)}{dd_i}\bigg|\leq \bigg(\sum_{i=1}^k d_i {}-1\bigg)\sqrt{q}.
$$
\end{thm}

\noindent
We remark that the assumption that no two~$v_i$ are Galois conjugates is necessary. Indeed, if~$v_i$ and~$v_j$ are Galois conjugates, then for $x \in \F_q$, it is easy to verify that $(x-v_i)^{(q^n-1)/d}=1$ if and only if $(x-v_j)^{(q^n-1)/d}=1$. 

Theorem~\ref{thm:main} recovers a few known results in the literature. In particular, if $n=1$, then $d_i=1$ for all $i$ and Theorem~\ref{thm:main} recovers Lemma~\ref{lem1}. Also, in the special case $d=2$, $k=1$, $n$ even, and $v_1 \notin \F_{q^{n/2}}$, Theorem~\ref{thm:main} shows that the number of solutions is $q/2+O(n\sqrt{q})$, which recovers a result of Hirschfeld and Sz{\H{o}}nyi \cite{HS90}; moreover, they showed this special case already has some nice applications in finite geometry.

We also prove the following theorem, which complements Theorem~\ref{thm:main} when $n=2$ and $q \not \equiv 1 \pmod d$. 

\begin{thm}\label{thm:main2}
Let $d \geq 2$ and let $q^2 \equiv 1 \pmod d$. Choose $v_1, v_2, \ldots, v_k \in \F_{q^2} \setminus \F_q$, no two of which are Galois conjugates. Let $M$ be the number of solutions $x$ in $\F_q$ to the system of equations $(x-v_i)^{(q^2-1)/d}=1$ for $1 \leq i \leq k$. Then
$$
\bigg |M- \frac{q}{d^k}\bigg|\leq (2k-1)\sqrt{q}.
$$    
\end{thm}

Note that if $q \equiv -1 \pmod d$, then we always have $(x-v)^{(q^2-1)/d}=1$ when $v,x \in \F_q$ and $x \neq v$. Also note that the condition that no two~$v_i$ are Galois conjugates is necessary, for the same reason as in Theorem~\ref{thm:main}. We remark that some related results have appeared in \cite[Corollaries~2.4 and~2.5]{W97}. However, the statements of those two corollaries omitted some necessary hypotheses\footnote{Private communication with Daqing Wan. For example, \cite[Corollary 2.4]{W97} does not apply to the case $m=2, n=1$ and $f_1(T)=(T-a)(T-a^q)$ for $a \in \F_{q^2} \setminus \F_q$.}. In Section~\ref{sec3}, we shall state and prove a corrected version for the sake of completeness. In particular, we will prove Theorem~\ref{thm:main3} and then deduce Theorem~\ref{thm:main2} as a consequence. One can use Theorem~\ref{thm:main3} to prove a more general result on the distribution of power residues over shifted subfields. 

Before discussing the applications of Theorem~\ref{thm:main} and Theorem~\ref{thm:main2}, we introduce some necessary terminology. Generalized Paley graphs are well-studied Cayley graphs, first introduced by Cohen \cite{SC} in 1988 and reintroduced by several groups of authors since then. Let $d\ge2$ be an integer and~$q$ a prime power such that $q \equiv 1 \pmod {2d}$. The {\em $d$-Paley graph} on~$\F_q$, denoted $GP(q,d)$, is the graph with vertex set~$\F_q$ such that two vertices are adjacent if their difference is a $d$-th power in~$\F_q^*$. Note that the condition $q \equiv 1 \pmod {2d}$ guarantees the graph is undirected and non-degenerate; see for example \cite[Section 4]{SC}. The well-known {\em Paley graphs} are simply $2$-Paley graphs. We refer to \cite{BEHW96, GKSV18, GMS22, GSY23, Y22, Y23} for extensive discussions on different constructions of maximal cliques in generalized Paley graphs; we also refer to \cite[Proposition 2]{S92} for a nice connection between maximal cliques in Paley graphs and minimal blocking sets.

Recall that a {\em clique}~$C$ in a graph~$G$ is a subset of vertices of $G$ such that every two distinct vertices in $C$ are adjacent. A {\em maximal clique} is a clique that cannot be extended by adding a new vertex. Every {\em maximum clique} (a clique of highest cardinality) is a maximal clique, but there can be smaller maximal cliques as well. As an application of Theorem~\ref{thm:main}, we construct new families of maximal cliques in generalized Paley graphs of the form $GP(q^d,d)$, where~$q$ is an odd prime power such that $q \equiv 1 \pmod d$. It is known that if $q>(d-1)^2$, then the subfield $\F_q$ forms a maximal clique in such a graph \cite[Theorem 1.2]{Y23}. Our next result constructs many new maximal cliques with smaller sizes. To state the theorem, recall that for a positive integer $m$, its \emph{radical} $\rad(m)$ is defined to be the product of the distinct prime divisors of $m$. In particular, $\rad(m)\mid\rad(d)$ if and only if every prime dividing~$m$ also divides~$d$. We let $\log_r$ denote the base-$r$ logarithm.

\begin{thm}\label{thm:newmaxclique}
Let $d \geq 2$ and let $r$ be the smallest prime divisor of $d$. Let $m$ be a positive integer such that $\rad(m) \mid \rad(d)$. If $q \equiv 1 \pmod d$ is an odd prime power such that
$$
q>(8\log_r m+4)d^2m^2,
$$
then there is a maximal clique $C$ in $GP(q^d,d)$ with size $|C|=\frac{q}{m}+O(d\log m\sqrt{q})$. More precisely,
$$
\frac{q}{m}-d\log_r m \cdot \sqrt{q}\leq |C|\leq \frac{q}{m}+d\log_r m \cdot (\sqrt{q}+1).
$$    
\end{thm}

In particular, if $d \geq 2$ is fixed and $m$ is a positive integer such that $\rad(m) \mid \rad(d)$, then there is a maximal clique of size approximately $q/m$ in $GP(q^d,d)$, provided that $q \equiv 1 \pmod d$ is sufficiently large. (We believe that these are the only possible approximate sizes of maximal cliques; see Conjecture~\ref{spectrum conj} for a more precise statement.)  When $d=2$ and $m$
is a power of $2$, this recovers a result of Hirschfeld and Sz{\H{o}}nyi on maximal cliques in Paley graphs of square order, which they did not state explicitly but is implicit in their constructions of minimal blocking sets \cite{HS91, S92}. It would be interesting to see if Theorem~\ref{thm:newmaxclique} has some applications in finite geometry.

For our second application, we focus on generalized Paley graphs of the form $GP(q^2,d)$, where~$q$ is an odd prime power such that $d \mid (q+1)$. Such generalized Paley graphs are of particular interest since the subfield $\F_q$ forms a maximum clique, and it is known that the only maximum clique containing~$0$ and~$1$ is the subfield~$\F_q$; see Blokhuis~\cite{B84} for Paley graphs (that is, when $d=2$) and Sziklai~\cite{S99} for general~$d$. This characterization of maximum cliques is also known as the Erd\H{o}s--Ko--Rado (EKR) theorem for the graph $GP(q^2,d)$, in the sense that all maximum cliques are given by affine translates of the subfield $\F_q$; see related discussions in \cite{AY22, Y24} and \cite[Section~5.9]{GM15}.  

Baker, Ebert, Hemmeter,
and Woldar \cite{BEHW96} constructed maximal cliques that are not maximum in the Paley graph $GP(q^2,2)$ as the following. Pick an element $\alpha \in \F_{q^2}\setminus \F_q$, and consider the clique~$C$ obtained by $\alpha$ and its \emph{$\F_q$-neighborhood}, that is,
$$
C=\{\alpha\} \cup \{x \in \F_q: \alpha-x \in (\F_{q^2}^*)^2\}.
$$
They showed that $C$ is a maximal clique of size $\frac{q+1}{2}$ if $q \equiv 1 \pmod 4$, and $C \cup \{\alpha^q\}$ is a maximal clique of size $\frac{q+3}{2}$ if $q \equiv 3 \pmod 4$. Their construction is also known as the \emph{$(\F_q, \alpha)$-construction} \cite{GKSV18,GMS22}. Analogously, Goryainov, Shalaginov, and Yip~\cite{GSY23} considered a similar construction of cliques in $GP(q^2,d)$, where $d \mid (q+1)$ and $d \geq 3$. Let $u \in \F_{q^2} \setminus \F_q$ and let $N(u)$ be the $\F_q$-neighborhood of $u$ in $GP(q^2,d)$. Let $C_u=N(u)\cup \{u\}$ if 
$d \nmid \frac{q+1}{2}$, and $C_u=N(u) \cup \{u,u^q\}$ if $d \mid \frac{q+1}{2}$, which they described~\cite[Proposition 4.6]{GSY23} as a clique from the $(\F_q, \alpha)$-construction. They conjectured that such a clique $C_u$ is maximal if $3 \leq d \leq \frac{q+1}{3}$ and $p \nmid (d-1)$, where $p$ is the characteristic of the field $\F_q$ \cite[Conjecture 4.7]{GSY23}; see \cite[Corollary 3.4]{GSY23} and \cite[Example 5.9]{GSY23} for the necessity of these two additional assumptions. Some partial progress of the conjecture can be found in \cite[Section 5]{GSY23}. In particular, they proved the conjecture when $\gcd(q-1,\frac{q+1}{d}-2)\in \{1,2\}$ \cite[Corollary 5.4]{GSY23}. We are able to confirm their conjecture when $q$ is sufficiently large compared~to~$d$.

\begin{thm}\label{thm:GP}
Let $d \geq 3$. If $q \equiv -1 \pmod d$ is an odd prime power such that  $q>10d^4/(d-1)^2$, then in $GP(q^2,d)$, cliques obtained from the $(\F_q, \alpha)$-construction are maximal. More precisely, if $u \in \F_{q^2}\setminus \F_q$ and $N(u)$ is the $\F_q$-neighborhood of $u$ in $GP(q^2,d)$, then the following statements hold: 
\begin{enumerate}
    \item[(a)] If $d \nmid \frac{q+1}{2}$, then $N(u) \cup \{u\}$ forms a maximal clique of size $\frac{q+1}{d}$ in $GP(q^2,d)$.
    \item[(b)] If $d \mid \frac{q+1}{2}$, then $N(u) \cup \{u,u^q\}$ forms a maximal clique of size $\frac{q+d+1}{d}$ in $GP(q^2,d)$.
\end{enumerate}    
\end{thm}

Our techniques extend to a larger family of Cayley graphs, and in particular to Peisert graphs. Let $p \equiv 3 \pmod 4$ be a prime and $q=p^r$ with $r$ even. The {\em Peisert graph} of order $q=p^r$, denoted $P^*_q$, is defined to be the graph with vertex set $\F_{q}$ such that two vertices are adjacent if their difference belongs to the set $S=\{g^j\colon j \equiv 0,1 \pmod 4\}$, where $g$ is a primitive root of the field~$\F_q$. Note the structure of the graph does not depend on the choice of the primitive root. Peisert~\cite{P01} showed that the only self-complementary symmetric graphs are Paley graphs, Peisert graphs, and an exceptional graph of order $529$. Asgarli and Yip \cite[Theorem 1.4]{AY22} showed an analgoue of the EKR theorem for the Peisert graph $P_{q^2}^*$ with $q \equiv 3 \pmod 4$; more precisely, the only maximum clique containing $0,1$ in $P_{q^2}^*$ is the subfield~$\F_q$ if $q=p^n$ and $p>8.2n^2$, where~$p$ is the characteristic of $\F_q$. Sergey Goryainov conjectured\footnote{Private communication} that $(\F_q, \alpha)$-constructions also give maximal cliques in $P_{q^2}^*$ with $q \equiv 3 \pmod 4$ and $q \geq 7$, based on the similarity between Paley graphs and Peisert graphs. We confirm his conjecture using a similar method.

\begin{thm}\label{thm:P*}
Let $q\equiv 3 \pmod 4$ be a prime power such that $q \geq 7$. Let $u \in \F_{q^2}\setminus \F_q$ and let $N(u)$ be the $\F_q$-neighborhood of $u$ in $P_{q^2}^*$. Then $N(u) \cup \{u\}$ forms a maximal clique of size $\frac{q+1}{2}$ in $P_{q^2}^*$.
\end{thm}

\medskip

The paper is organized as follows.
In Section~\ref{sec2}, we prove Theorem~\ref{thm:main}. In Section~\ref{sec3}, we establish some estimates on character sums over subfields and prove Theorem~\ref{thm:main2}. In Section~\ref{sec4}, we construct new maximal cliques in generalized Paley graphs and prove Theorem~\ref{thm:newmaxclique}. In Section~\ref{sec5}, we study the maximality of cliques obtained from the $(\F_q,\alpha)$-construction and prove Theorem~\ref{thm:GP} as well as Theorem~\ref{thm:P*}. We end the paper with some remarks and open questions in Section~\ref{sec6}.

\section{Proof of Theorem~\ref{thm:main}}\label{sec2}

Consider the norm map $N_{\F_{q^n}/\F_q}\colon \F_{q^n} \to \F_q$ of the field extension $\F_{q^n}/\F_q$; explicitly, 
$$N_{\F_{q^n}/\F_q}(x)=\prod_{j=0}^{n-1} x^{q^j}=x^{\frac{q^n-1}{q-1}}.$$
The next lemma provides a criterion to determine whether an element $x$ in $\F_{q^n}$ is a $d$-th power based on its norm $N_{\F_{q^n}/\F_q}(x)$. To discuss $d$-th powers in the base field $\F_q$, we need to further assume $d \mid (q-1)$.

\begin{lem}\label{lem: norm_reduction}
Assume $d \mid (q-1)$. Let $x \in \F_{q^n}$. Then $x$ is a $d$-th power in $\F_{q^n}$ if and only if $N_{\F_{q^n}/\F_q}(x)$ is a $d$-th power in $\F_q$.
\end{lem}

\begin{proof}
The case $x=0$ is trivial. Next, assume $x \neq 0$. Let $g$ be a primitive root of $\F_{q^d}$.  Let $x=g^k$. Then $x$ is a $d$-th power in $\F_{q^n}$ if and only if $d \mid k$. Note that $N_{\F_{q^n}/\F_q}(g)=g^{\frac{q^n-1}{q-1}}$ is a primitive root of $\F_q$. Thus, $N_{\F_{q^n}/\F_q}(x)=(g^{\frac{q^n-1}{q-1}})^k$ is a $d$-th power in $\F_{q}$ if and only if $d \mid k$, as claimed.
\end{proof}

To prove Theorem~\ref{thm:main}, we also need the celebrated Weil bound; see for example \cite[Theorem~5.41]{LN97}.

\begin{lem}[Weil's bound]\label{lem:Weil}
Let $\chi$ be a multiplicative character of $\F_q$ of order $d>1$, and let $f \in \F_q[z]$ be a monic polynomial of positive degree that is not a $d$-th power of a polynomial. 
Let~$m$ be the number of distinct roots of~$f$ in its
splitting field over~$\F_q$. Then for any $a \in \F_q$,
$$\bigg |\sum_{z\in\mathbb{F}_q}\chi\big(af(z)\big) \bigg|\le(m-1)\sqrt q \,.$$
\end{lem}

Now we are ready to present the proof of Theorem~\ref{thm:main}.

\begin{proof}[Proof of Theorem~\ref{thm:main}]
Let $g_i(z)$ be the minimal polynomial of $v_i$ over $\F_q$; then $g_i(z)$ has degree $d_i$ and 
$$
g_i(z)=\prod_{j=0}^{d_i-1} (z-v_i^{q^j}).
$$ 
Let $f_i=g_i^{n/d_i}$. Then $f_i$ is of degree $n$, and each root of $f_i$ is a Galois conjugate of $v_i$ with multiplicity $n/d_i$. 

Note that for each $z \in \F_q$, since $v_i^{q^{d_i}}=v_i$, we have
$$
f_i(z)=\prod_{j=0}^{d_i-1} (z-v_i^{q^j})^{n/d_i}=\prod_{j=0}^{n-1} (z-v_i^{q^j})=\prod_{j=0}^{n-1} (z-v_i)^{q^j}=N_{\F_{q^n}/\F_q}(z-v_i).
$$

Let $\chi$ be a multiplicative character of $\F_q$, with order $d$. If $x \in \F_q$, then $(x-v_i)^{(q^n-1)/d}=1$ for each $i$ if and only if $x-v_i$ is a $d$-th power in $\F_{q^n}$ for each~$i$, if and only if $f_i(x)=N_{\F_{q^n}/\F_q}(x-v_i)$ is a $d$-th power in $\F_q$ for each~$i$ (by Lemma~\ref{lem: norm_reduction}), if and only if $\chi(f_i(x))=1$ for each~$i$.

By the orthogonality relations, $\frac{1}{d} \sum_{j=0}^{d-1} \chi^j$ is the indicator function of $d$-th powers in $\F_q^*$. Therefore, by the above discussion, the number of solutions to the given system of equations is
\begin{equation} \label{eq: common}
 M=\sum_{x \in \F_q} \prod_{i=1}^k \bigg(\frac{1}{d} \sum_{j=0}^{d-1} \chi^j(f_i(x)) \bigg)=\frac{1}{d^k} \sum_{0 \leq j_1, j_2, \ldots, j_k \leq d-1}\sum_{x \in \F_q} \chi\bigg(\prod_{i=1}^k f_i^{j_i}(x)\bigg).   
\end{equation}

Since $v_i$ and $v_j$ are not Galois conjugates whenever $i \neq j$, it follows that $f_1, f_2, \ldots, f_k$ are pairwise coprime. It follows that $\prod_{i=1}^k f_i^{j_i}(z)$ is a $d$-th power of a polynomial if and only if $d \mid j_i \cdot n/d_i$ for each $i$, or equivalently,
$$
\frac{dd_i}{\gcd(dd_i,n)} \mid j_i.
$$
Thus, the number of $k$-tuples $(j_1,j_2, \ldots, j_k)$ such that $0 \leq j_i \leq d-1$ and $\prod_{i=1}^k f_i^{j_i}(z)$ is a $d$-th power of a polynomial is 
$$
\prod_{i=1}^k  \frac{d\gcd(dd_i,n)}{dd_i} =\prod_{i=1}^k  \frac{\gcd(dd_i,n)}{d_i} ,
$$
and the contribution of these $k$-tuples (corresponding to trivial character sums) to the sum~\eqref{eq: common} is
$$
\frac{q}{d^k} \cdot \prod_{i=1}^k  \frac{\gcd(dd_i,n)}{d_i}= q \cdot \prod_{i=1}^k  \frac{\gcd(dd_i,n)}{dd_i}.
$$
When $\prod_{i=1}^k f_i^{j_i}(z)$ is not a $d$-th power of a polynomial, the  number of distinct roots of the polynomial $\prod_{i=1}^{k} f_i(z)$ in its splitting field is at most $\sum_{i=1}^k d_i$, and thus Weil's bound (Lemma~\ref{lem:Weil}) implies that 
$$
\bigg|\sum_{x \in \F_q} \chi\bigg(\prod_{i=1}^k f_i^{j_i}(x)\bigg)\bigg| \leq \bigg(\sum_{i=1}^k d_i {}-1\bigg)\sqrt{q}.
$$
We conclude that
$$
\bigg|M-q \cdot \prod_{i=1}^k  \frac{\gcd(dd_i,n)}{dd_i}\bigg| \leq \frac{1}{d^k} \cdot d^k \cdot \bigg(\sum_{i=1}^k d_i {}-1\bigg)\sqrt{q}=\bigg(\sum_{i=1}^k d_i {}-1\bigg)\sqrt{q}
$$
as required.
\end{proof}

\section{Proof of Theorem~\ref{thm:main2}}\label{sec3}

Similar to the discussions in the previous section, the proof of Theorem~\ref{thm:main2} will be based on character sums. Clearly, we need to estimate character sums of the form $\sum_{x \in \F_q} \chi(f(x))$, where $\chi$ is a multiplicative character of $\F_{q^2}$. The key idea is to convert the desired character sum over finite fields to an equivalent character sum over function fields. 

Let $\F_q[T]$ be the polynomial ring in variable $T$ over $\F_q$. Let $f \in \F_q[T]$ be a non-constant polynomial. A \emph{Dirichlet character modulo $f$}, usually denoted by $\chi_f$, is a character on the multiplicative group $(\F_q[T]/f\F_q[T])^*$, which can be extended to a function on $\F_q[T]$ by setting $\chi_f(g)=\chi_f(g \bmod f)$ if $f$ and $g$ are coprime, and $\chi_f(g)=0$ otherwise. We refer to \cite[Section 4]{R02} for more background.

We need the following Weil bound on character sums over monic linear polynomials; see for example \cite[Theorem 2.1]{W97}.

\begin{lem}\label{lem:FF}
Let $f \in \F_q[T]$ be a polynomial with degree $n \geq 1$, and let $\chi_f$ be a non-trivial Dirichlet character modulo $f$. Then
$$
\bigg|\sum_{a \in \F_q} \chi_{f}(T-a)\bigg|\leq (n-1)\sqrt{q}.
$$
\end{lem}

From this lemma, we can deduce the following theorem.

\begin{thm}\label{thm:main3}
Let $f_1, f_2,\ldots, f_k$ be monic irreducible polynomials in $\F_{q^n}[T]$, no two of which are Galois conjugates over $\F_q$.  Let $\chi_1,\chi_2,\ldots, \chi_k$ be multiplicative character of $\F_{q^n}$. Assume that there exists $1\leq i \leq k$ such that~$f_i$ has a root~$\xi_i$ such that $\chi_i$ is not identically $1$ on the set $N_{\F_{q^n}[\xi_i]/\F_{q^n}}(\F_q[\xi_i]) \setminus \{0\}$. Let $b_i$ be the degree of $f_i(T)$ and $c_i$ be the number of conjugates of $f_i(T)$ over $\F_q$. Then 
$$\bigg |\sum_{a\in\mathbb{F}_q}\prod_{i=1}^k\chi_i\big(f_i(a)\big) \bigg|\le\bigg(\sum_{i=1}^{k} b_ic_i {}-1\bigg)\sqrt q \,.$$   
In particular, if $m = \sum_{i=1}^{k} b_i$ is the sum of the degree of the~$f_i$, then 
$$\bigg |\sum_{a\in\mathbb{F}_q}\prod_{i=1}^k\chi_i\big(f_i(a)\big) \bigg|\le(mn-1)\sqrt q \,.$$    
\end{thm}
\begin{proof}
For each $1 \leq i \leq k$, let $\xi_i$ be a root of $f_i$. Since $f_i$ is a monic irreducible polynomial in $\F_{q^n}[T]$, it follows that $f_i$ is the minimal polynomial of $\xi_i$ over $\F_{q^n}$ and thus 
\begin{equation}\label{eq:f_i(a)}
f_i(a)=N_{\F_{q^n}[\xi_i]/\F_{q^n}}(a-\xi_i)    
\end{equation}
for $a \in \F_q$. For each $1 \leq i \leq k$, let $F_i$ be the product of conjugates of $f_i(T)$ over $\F_q$. Then, $F_1, F_2, \ldots, F_k$ are defined and irreducible over $\F_q$. Since no two of $f_1, f_2,\ldots, f_k$ are Galois conjugates over $\F_q$, it follows that $F_1, F_2, \ldots, F_k$ are coprime. Let $F=\prod_{i=1}^k F_i$. 
Since $\deg(F_i)=b_ic_i$, it follows that $\deg(F)=\sum_{i=1}^k b_ic_i$.

For $g \in \F_q[T]$, define 
\begin{equation}\label{eq:chi_{F_i}}
\chi_{F_i}(g)=\chi_i(N_{\F_{q^n}[\xi_i]/\F_{q^n}}(g(\xi_i))).
\end{equation}
Note that $\chi_{F_i}$ is a Dirichlet character modulo $F_i$. Note also that, as $g$ runs over $\F_q[T]$, $g(\xi_i)$ runs over $\F_q[\xi_i]$, and thus the norm $N_{\F_{q^n}[\xi_i]/\F_{q^n}}(g(\xi_i))$ runs over the set $N_{\F_{q^n}[\xi_i]/\F_{q^n}}(\F_q[\xi_i])$. Thus, if $\chi_i$ is not identically $1$ on the set $N_{\F_{q^n}[\xi_i]/\F_{q^n}}(\F_q[\xi_i]) \setminus \{0\}$, then $\chi_{F_i}$ is non-trivial by definition. Therefore, by the given assumption, at least one of the characters $\chi_{F_1}, \chi_{F_2}, \ldots, \chi_{F_k}$ is non-trivial. Since $F_1, F_2, \ldots, F_k$ are irreducible coprime polynomials, by the Chinese remainder theorem, the product $\chi_F:=\prod_{i=1}^k \chi_{F_i}$ is a non-trivial Dirichlet character modulo $F$. Therefore, by Lemma~\ref{lem:FF} and equations~\eqref{eq:f_i(a)} and~\eqref{eq:chi_{F_i}},
$$
\bigg |\sum_{a\in\mathbb{F}_q}\prod_{i=1}^k\chi_i\big(f_i(a)\big) \bigg|=\bigg |\sum_{a\in\mathbb{F}_q}\prod_{i=1}^k\chi_{F_i}(a-T) \bigg|=\bigg |\sum_{a\in\mathbb{F}_q}\chi_{F}(a-T) \bigg|\leq \bigg(\sum_{i=1}^k b_ic_i-1\bigg)\sqrt{q},
$$
as required. Since $c_i \leq n$, we always have $\sum_{i=1}^k b_ic_i \leq n \sum_{i=1}^k b_i=mn$.
\end{proof}

\begin{rem}\label{rem:trivial}
From the above proof, we can infer that the assumption that ``there exists $1\leq i \leq k$ such that~$f_i$ has a root~$\xi_i$ such that $\chi_i$ is not identically $1$ on the set $N_{\F_{q^n}[\xi_i]/\F_{q^n}}(\F_q[\xi_i]) \setminus \{0\}$" in Theorem~\ref{thm:main3} is equivalent to the character sum being non-trivial. In particular, if this condition does not hold, then the character sum in Theorem~\ref{thm:main3} is~$q$. 
\end{rem}

Next, we use Theorem~\ref{thm:main3} to deduce Theorem~\ref{thm:main2}.

\begin{proof}[Proof of Theorem~\ref{thm:main2}]

For each $1 \leq i \leq k$, let $f_i(T)=T-v_i \in \F_{q^2}[T]$. By the assumption, no two of $f_1,\ldots, f_k$ are Galois conjugates of each other. Note that for each $1 \leq i \leq k$, we have $N_{\F_{q^2}[v_i]/\F_{q^2}}(\F_q[v_i])=N_{\F_{q^2}/\F_{q^2}}(\F_{q^2})=\F_{q^2}$ since $v_i \in \F_{q^2} \setminus \F_q$.

Let $\chi$ be a multiplicative character in $\F_{q^2}$, with order $d$. 
Similar to the proof of Theorem~\ref{thm:main}, the number of solutions to the given system of equations is
\begin{equation} \label{eq: common2}
 M=\frac{1}{d^k} \sum_{0 \leq j_1, j_2, \ldots, j_k \leq d-1}\sum_{x \in \F_q} \chi\bigg(\prod_{i=1}^k f_i^{j_i}(x)\bigg).   
\end{equation}  
When $j_1=j_2=\cdots=j_k=0$, the character sum contributes to $\frac{q}{d^k}$ to $M$. In all other cases, at least one of $\chi^{j_1},\chi^{j_2}, \ldots, \chi^{j_k}$ is a nontrivial character since $\chi$ has order $d$, and Theorem~\ref{thm:main3} implies that
$$
\bigg|\sum_{x \in \F_q} \chi\bigg(\prod_{i=1}^k f_i^{j_i}(x)\bigg)\bigg|=\bigg|\sum_{x \in \F_q} \prod_{i=1}^k \chi^{j_i} (f_i(x))\bigg|\leq (2k-1)\sqrt{q}.
$$
Therefore, summing over $k$-tuples $(j_1,j_2, \ldots, j_k)$ in equation~\eqref{eq: common2}, we conclude that
$$
\bigg|M-\frac{q}{d^k}\bigg|\leq (2k-1)\sqrt{q},
$$
as required.
\end{proof}

\begin{rem}
Our proof of Theorem~\ref{thm:main2} in this section employed character sums over finite fields. One can give another proof of Theorem~\ref{thm:main} using Theorem~\ref{thm:main3} and Remark~\ref{rem:trivial}, which employ character sums over function fields. To see that, we again set $f_i(T)=T-v_i \in \F_{q^n}[T]$ for each $1\leq i \leq k$, and equation~\eqref{eq: common2} still applies. Since $v_i$ has degree $d_i$ over $\F_q$, the number of conjugates of $f_i$ over $\F_q$ also is $d_i$. Note that we have $N_{\F_{q^n}[v_i]/\F_{q^n}}(\F_q[v_i]) \setminus \{0\}=\F_{q^{d_i}}^*$, and thus the main term $q \prod_{i=1}^k  \frac{\gcd(dd_i,n)}{dd_i}$ in Theorem~\ref{thm:main} follows from Remark~\ref{rem:trivial}, and the error term $(\sum_{i=1}^k d_i{}-1)\sqrt{q}$ in Theorem~\ref{thm:main} comes from Theorem~\ref{thm:main3}. While the exact machinery of the two proofs are different, there are definite similarities and the two proofs are probably essentially equivalent.
\end{rem}

Next, we state and prove a corrected version of \cite[Corollary 2.4]{W97}. Any monic polynomial $f \in \F_{q^n}[T]$ with positive degree can be factorized as the product of monic irreducible polynomials in $\F_{q^n}[T]$:
$$
f=\prod_{i=1}^{s} f_i^{t_i},
$$
where the~$f_i$ are distinct irreducible polynomials. Let $\xi$ be a root of $f_1$. By relabelling, we may assume $f_1,f_2, \ldots, f_{r}$ are the conjugates of $f_1$ over $\F_q$. Let $\sigma\colon x \mapsto x^q$ be the Frobenius map defined on~$\F_{q^n}$; note that $\sigma$ naturally induces a map on $\F_{q^n}[T]$. Then for each $1 \leq i \leq r$, we can write $f_i=\sigma^{\alpha_i}(f_1)$ for some $0\le\alpha_i\le n-1$. We define the \emph{total multiplicity} of $\xi$ to be
$$
m=\sum_{i=1}^{r} t_iq^{\alpha_i},
$$
which is the sum of weighted multiplicities among conjugates of $\xi$.
Observe that if $\chi$ is a character over $\F_{q^n}$, then we have
$$
\sum_{a \in \F_q} \chi(f(a))=\sum_{a \in \F_q} \prod_{i=1}^s \chi^{t_i}(f_i(a)).
$$
Note that for $1 \leq i \leq r$, and each $a \in \F_q$, we have $f_i(a)=(\sigma^{\alpha_i}(f_1))(a)=(f_1(a))^{q^{\alpha_i}}$. Thus, for $a \in \F_q$, we have
$$
\prod_{i=1}^{r} \chi^{t_i}(f_i(a))=\prod_{i=1}^{r} \chi^{t_iq^{\alpha_i}}(f_1(a))=\chi^m(f_1(a)).
$$
Using this observation, Theorem~\ref{thm:main3} implies the following corollary immediately. 
\begin{cor}
Let $f_1, f_2,\ldots, f_k$ be monic polynomials in $\F_{q^n}[T]$, no two of which share roots that are Galois conjugates over $\F_q$. Let $m$ be the degree of the largest squarefree divisor of $\prod_{i=1}^kf_i$. Let $\chi_1,\chi_2,\ldots, \chi_k$ be multiplicative characters of $\F_{q^n}$. Assume that there exists $1\leq i \leq k$ such that~$f_i$ has a root~$\xi_i$ of total multiplicity~$m_i$ such that $\chi_i^{m_i}$ is not identically~$1$ on the set $N_{\F_{q^n}[\xi_i]/\F_{q^n}}(\F_q[\xi_i]) \setminus \{0\}$. Then
$$\bigg |\sum_{a\in\mathbb{F}_q}\prod_{i=1}^k\chi_i\big(f_i(a)\big) \bigg|\le(mn-1)\sqrt q .$$
\end{cor}

\section{New constructions of maximal cliques in generalized Paley graphs}\label{sec4}

To construct maximal cliques in $GP(q^d,d)$, we are led to consider the local behavior of the graph $GP(q^d,d)$ in a subfield, equivalently, the structure of the subgraph induced by a subfield.

\begin{lem}\label{lem:subgraph}
Let $d \geq 2$ and let $q \equiv 1 \pmod d$ be an odd prime power. Let $d'$ be a divisor of $d$ that is greater than $1$. Then the subgraph of $GP(q^d,d)$ induced by the subfield $\F_{q^{d'}}$ is the same as $GP(q^{d'},d')$.
\end{lem}
\begin{proof}
Let $g$ be a primitive root of $\F_{q^d}$. Then $g^{(q^d-1)/(q^{d'}-1)}$ is a primitive root of $\F_{q^{d'}}$; moreover, since $q \equiv 1 \pmod d$, it follows that
$$
\frac{q^d-1}{q^{d'}-1}=\sum_{j=0}^{d/d'-1} q^{d'j} \equiv \frac{d}{d'} \pmod d.
$$
Thus, $\gcd(d,\frac{q^d-1}{q^{d'}-1})=\frac{d}{d'}$. Let $x \in \F_{q^{d'}}^*$. Then $x$ is a $d$-th power in $\F_{q^d}$ if and only if $x$ is a $d'$-th power in $\F_{q^{d'}}$. The conclusion follows.
\end{proof}

As preparation for proving Theorem~\ref{thm:newmaxclique}, in the next proposition, we establish the existence of cliques of small size with prescribed degrees.

\begin{prop}\label{prop:predeg}
Let $d \geq 2$ and let $q \equiv 1 \pmod d$ be an odd prime power. Let $r$ be the smallest prime divisor of $d$. Let $d_1,d_2,\ldots, d_k$ be positive integers such that $d_1>1$ and $d_1\mid d_2 \mid \cdots \mid d_k \mid d$. Assume that
$$
q^r> \max \{(d+(k-1)r^{k-1})^2,e^{2(k-1)}\}.
$$
Then there is a clique $C=\{v_1,v_2, \ldots, v_k\}$ in $GP(q^d,d)$ such that $v_i$ has degree $d_i$ over $\F_q$ for each $1 \leq i \leq k$, and no two vertices in $C$ are Galois conjugates with respect to the field extension $\F_{q^d}/\F_q$.
\end{prop}
\begin{proof}
We build such a clique inductively.  We first pick an arbitrary $v_1 \in \F_{q^{d_1}}$ such that $v_1$ has degree $d_1$ over $\F_q$. 

Given $2 \leq j \leq k$, suppose we have constructed a clique $C_{j-1}=\{v_1, v_2, \ldots, v_{j-1}\}$ in $GP(q^d,d)$ such that $v_i$ has degree $d_i$ over $\F_q$ for each $1 \leq i \leq j-1$, and no two vertices in $C_{j-1}$ are Galois conjugates with respect to the field extension $\F_{q^d}/\F_q$. We need to find $v_j \in \F_{q^{d_j}}$ with degree $d_j$, not a Galois conjugate of any of $v_1, v_2, \ldots, v_{j-1}$, such that $C_{j-1} \cup \{v_j\}$ forms a clique in $GP(q^d,d)$.

Note that $v_1, v_2, \ldots, v_{j-1},v_j \in \F_{q^{d_j}}$. Thus, by Lemma~\ref{lem:subgraph}, $C_{j-1} \cup \{v_j\}$ forms a clique in $GP(q^d,d)$ if and only if 
$C_{j-1} \cup \{v_j\}$ forms a clique in $GP(q^{d_j},d_j)$, or equivalently, if and only if $(v_j-v_i)^{(q^{d_j}-1)/d_j}=1$ for each $1 \leq i \leq j-1$. By Lemma~\ref{lem1}, the number of such $v_j$ is at least $q^{d_j}/d_j^{j-1}-(j-1)\sqrt{q^{d_j}}$. On the other hand, note that the number of elements in $\F_{q^{d_j}}$ that are Galois conjugates of one of $v_1, v_2, \ldots, v_{j-1}$ is at most $(j-1)d_j$. Therefore, if 
\begin{equation}\label{eq:j}
\frac{q^{d_j}}{d_j^{j-1}}>(j-1)\sqrt{q^{d_j}}+(j-1)d_j    
\end{equation}
holds, then we can find a desired $v_j$ and achieve our goal.

It remains to show that the given assumption $q^r> \max \{(d+(k-1)r^{k-1})^2,e^{2(k-1)}\}$ implies the inequality~\eqref{eq:j} for each $2 \leq j \leq k$. Consider the function $s(t)=t \log q-2(k-1)\log t$, where $t \geq r$ is a real number. Since $q^r> e^{2(k-1)}$, it follows that $s'(t)=\log q-2(k-1)/t>0$ when $t\geq r$. Therefore, $s(t) \geq s(r)$ when $t \geq r$, or equivalently $\sqrt{q^{t}}/t^{k-1} \geq \sqrt{q^{r}}/r^{k-1}$ when  $t \geq r$. Therefore, given $q^r>(d+(k-1)r^{k-1})^2$ where~$r$ is the smallest prime divisor of~$d$, we have for each $2 \leq j \leq k$ that
\begin{align*}
\frac{q^{d_j}}{d_j^{j-1}}-(j-1)\sqrt{q^{d_j}}-(j-1)d_j 
&\geq \frac{q^{d_j}}{d_j^{k-1}}-(k-1)\sqrt{q^{d_j}} -(k-1)d\\
&= (k-1)\sqrt{q^{d_j}} \bigg(\frac{\sqrt{q^{d_j}}}{(k-1)d_j^{k-1}}-1\bigg)-(k-1)d\\
&\geq (k-1)\sqrt{q^{r}} \bigg(\frac{\sqrt{q^{r}}}{(k-1)r^{k-1}}-1\bigg)-(k-1)d    \\
&>(k-1)^2r^{k-1} \bigg(\frac{d}{(k-1)r^{k-1}}\bigg)-(k-1)d=0,
\end{align*}
as required. 
\end{proof}

After one final elementary lemma, we will be ready to prove Theorem~\ref{thm:newmaxclique}. 

\begin{lem}\label{lem:div2}
Let $m,d\ge2$ be positive integers with $\rad(m)\mid \rad(d)$, and let~$r$ be the smallest prime divisor of~$m$. Then there exist positive integers~$k$ and $d_1, d_2, \ldots, d_k\ge r$ such that
$d_1d_2\cdots d_k=m$
and
$d_1\mid d_2 \mid \cdots \mid d_k \mid d$.
\end{lem}

\begin{proof}
If we let~$k$ be the largest exponent of any prime in the prime factorization of~$m$, it is easy to check that the integers
$d_i = \prod_{p^{k+1-i}\mid m} p$
have the asserted properties.
\end{proof}

\begin{proof}[Proof of Theorem~\ref{thm:newmaxclique}]
Note that the assumption $q \equiv 1 \pmod d$ guarantees that $\F_q$ forms a clique in $GP(q^d,d)$. If $m=1$, the subfield $\F_q$ forms a maximal clique in $GP(q^d,d)$ \cite[Theorem 1.2]{Y23} as mentioned in the introduction. Thus we may assume that $m\ge2$ in the following.

Since $\rad(m) \mid \rad(d)$, Lemma~\ref{lem:div2} enables us to choose positive integers $d_1, d_2, \ldots, d_k\ge r$ so that $m=d_1d_2\cdots d_k$ and $d_1\mid d_2 \mid \cdots \mid d_k \mid d$. Note that $m\ge r^k$ and thus $k \leq \log_r m$. It follows that
$$
q>(8\log_r m+4)d^2m^2\geq (8k+4)d^2m^2 \geq (8k+4)d^2r^{2k}>\max \{d+(k-1)r^{k-1},e^{k-1}\}.
$$
To see the last inequality, note that 
$(8k+4)d^2r^{2k}>d+(k-1)r^{2k}>d+(k-1)r^{k-1}$ and $r^{2k}\geq 4^k>e^{k-1}$. By Proposition~\ref{prop:predeg}, there exists a clique $D=\{v_1,v_2, \ldots, v_k\}$ in $GP(q^d,d)$ such that~$v_i$ has degree~$d_i$ over~$\F_q$ for each $1 \leq i \leq k$, and no two vertices in~$D$ are Galois conjugates with respect to the field extension $\F_{q^d}/\F_q$.

Recall that $\F_q$ is a clique. Observe that $C=D \cup D'$ also forms a clique in $GP(q^d,d)$, where
$$
D'=\{x \in \F_q: v_i-x \text{ is a $d$-th power in $\F_{q^d}$ for all $1 \leq i \leq k$}\}.
$$
Since $m=d_1d_2\cdots d_k$ and $\sum_{i=1}^k d_i \leq kd$, it follows from  Theorem~\ref{thm:main} that
\begin{equation}\label{eq:D'}
\frac{q}{m}-kd\sqrt{q} \leq |D'|\leq \frac{q}{m}+kd\sqrt{q}.    
\end{equation}
By definition, $C$ cannot be expanded to a larger clique by adding another vertex in $\F_q$. Assume that there exists $v \in \F_{q^d}\setminus \F_q$ such that $v$ is not a Galois conjugate of~$v_i$ with respect to the field extension $\F_{q^d}/\F_q$ for any $1 \leq i \leq k$, and yet $C \cup \{v\}$ is still a clique. It would then follow that
$$
D'=D' \cap \{x \in \F_q: v-x \text{ is a $d$-th power in $\F_{q^d}$}\}.
$$
Theorem~\ref{thm:main} then implies that
\begin{equation}\label{eq:D'ub}
|D'|\leq \frac{q}{mr}+(k+1)d\sqrt{q}.
\end{equation}
Comparing inequality~\eqref{eq:D'ub} with the lower bound on $|D'|$ in inequality~\eqref{eq:D'}, we obtain
$$
\frac{q}{m}-kd\sqrt{q} \leq \frac{q}{mr}+(k+1)d\sqrt{q},
$$
which implies that
$$
q\leq \frac{(2k+1)d^2m^2}{(1-\frac{1}{r})^2}\leq (8k+4)d^2m^2\leq (8\log_r m+4)d^2m^2,
$$
violating the assumption that $q>(8\log_r m+4)d^2m^2$.

Let $C'$ be a maximal clique in $GP(q^d,d)$ such that $C \subset C'$. Based on the above discussions, each element in $C' \setminus C$ is a Galois conjugate of $v_i$ for some $1 \leq i \leq k$. It follows that $|C|\leq |C'|\leq |C|+k(d-1)$, and thus inequality~\eqref{eq:D'} implies that
$$
\frac{q}{m}-d\log_r m \cdot \sqrt{q}\leq \frac{q}{m}-kd\sqrt{q}+k \leq |C'|\leq \frac{q}{m}+kd\sqrt{q}+kd \leq \frac{q}{m}+d\log_r m \cdot (\sqrt{q}+1),
$$
as required.
\end{proof}

\section{Maximality of the $(\mathbb{F}_q,\alpha)$-construction}\label{sec5}

In this section, we prove Theorem~\ref{thm:GP} and Theorem~\ref{thm:P*}. We start by presenting the proof of Theorem~\ref{thm:GP}.

\begin{proof}[Proof of Theorem~\ref{thm:GP}]
Recall that for each vertex $u \in \F_{q^2}$, $N(u)$ is the $\F_q$-neighborhood of $u$ in the $d$-Paley graph $GP(q^2,d)$, that is, 
$$
N(u)=\{x \in \F_q: u-x \text{ is a $d$-th power in } \F_{q^2}\}.
$$
By \cite[Proposition 4.6]{GSY23}, $|N(u)|=\frac{q+1}{d}-1$ if $u \in \F_{q^2}\setminus \F_q$. Since $d \mid (q+1)$, the subfield $\F_q$ forms a clique. Since $q>10d^4/(d-1)^2>10d^2$, we have
$$
9+\frac{6}{\sqrt{q}}+\frac{1}{q}<9+\frac{6}{\sqrt{10}d}+\frac{1}{40}<10.
$$
It follows that
$$
(3\sqrt{q}+1)^2=9q+6\sqrt{q}+1<10q<q^2\bigg(\frac{1}{d}-\frac{1}{d^2}\bigg)^2,
$$
and thus
\begin{equation}\label{eq:lbub}
\frac{q}{d^2}+3\sqrt{q}<\frac{q+1}{d}-1.    
\end{equation}

Let $u \in \F_{q^2} \setminus \F_q$. Then $u$ and $u^q$ are Galois conjugates. Note that $N(u)=N(u^q)$. Indeed, if $x \in \F_q$, then $u^q-x=u^q-x^q=(u-x)^q$ is a $d$-th power in $\F_{q^2}$ if and only if $u-x$ is a $d$-th power in $\F_{q^2}$. We also claim that that $u$ and $u^q$ are adjacent if and only if $d\mid \frac{q+1}{2}$. Let $\beta \in \F_q^*$ be a non-square in $\F_q$, and let $\alpha \in \F_{q^2}$ such that $\alpha^2=\beta$. Then $\{1,\alpha\}$ forms a basis of $\F_{q^2}$ over $\F_q$, and thus we can write $u=x+y\alpha$ for some $x,y \in \F_q$ with $y \neq 0$. Since $u^q=(x+y\alpha)^q=x^q+y^q\alpha^q=x-y\alpha$, we have $u-u^q=2y\alpha$ being a $d$-th power in $\F_{q^2}$ if and only if $\alpha$ is a $d$-th power in $\F_{q^2}$, if and only if $d\mid \frac{q+1}{2}$. Thus, $N(u) \cup \{u,u^q\}$ is a clique when $d \mid \frac{q+1}{2}$, and $N(u) \cup \{u,u^q\}$ is not a clique when $d\nmid \frac{q+1}{2}$.

Let $C_u=N(u) \cup \{u\}$ when $d\nmid \frac{q+1}{2}$, and $C_u=N(u) \cup \{u,u^q\}$ when $d\mid \frac{q+1}{2}$. Suppose that $C_u$ is not a maximal clique. Then based on the above discussion, there is $v \in \F_{q^2} \setminus \F_q$ such that $v \neq u^q$, and $N(u) \cup \{u,v\}$ remains to be a clique. It follows that $N(u) \cap N(v)=N(u)$. By Theorem~\ref{thm:main2} and inequality~\eqref{eq:lbub}, we have
$$
|N(u) \cap N(v)| \leq \frac{q}{d^2}+3\sqrt{q}<\frac{q+1}{d}-1=|N(u)|,
$$
a contradiction. We conclude that the desired clique $C_u$ is maximal.
\end{proof}

Finally, we use a similar approach to show the maximality of cliques from $(\F_q,\alpha)$-constructions in Peisert graphs.

\begin{proof}[Proof of Theorem~\ref{thm:P*}]
When $7\leq q \leq 79$, we have used SageMath to verify the statement of the theorem. In the following discussion, we assume that $q \geq 83$.

Let $g$ be a primitive root of $\F_{q^2}$ and let $H$ be the subgroup of $\F_{q^2}^*$ of index $4$. Let $S=H \cup gH$.  Note that two vertices are adjacent in the Peisert graph $P_{q^2}^*$ if and only if their difference is in $S$. Also, note that the $4$-Paley graph $GP(q^2,4)$ is a subgraph of 
the Peisert graph $P_{q^2}^*$. Since $4 \mid (q+1)$, the subfield $\F_q$ forms a clique in $GP(q^2,4)$, and thus in $P_{q^2}^*$. For each vertex $u \in \F_{q^2}$, let $N(u)$ be the $\F_q$-neighborhood of $u$ in $P_{q^2}^*$, that is, 
$$
N(u)=\{x \in \F_q: u-x \in S\}=\{x \in \F_q: u-x \in H\} \cup \{x \in \F_q: u-x \in gH\}:=N_1(u) \cup N_2(u).
$$
Note that $N_1(u)$ is exactly the $\F_q$-neighborhood of $u$ in $GP(q^2,4)$, and thus $|N_1(u)|=\frac{q+1}{4}-1$ if $u \notin \F_q$. On the other hand, it is known that Peisert graph $P_{q^2}^*$ is a $(q^2, \frac{q^2-1}{2}, \frac{q^2-5}{4}, \frac{q^2-1}{4})$-strongly regular graph with the smallest eigenvalue $\frac{-1-q}{2}$ \cite{P01}. Thus, the subfield $\F_q$ forms a maximum clique in $P_{q^2}^*$ and $|N(u)|=\frac{q+1}{2}-1$ for each $u \in \F_{q^2} \setminus \F_q$ by the Hoffman bound \cite[Proposition 1.3.2]{BCN89}. 

Let $u \in \F_{q^2} \setminus \F_q$. For the sake of contradiction, assume that $N(u)\cup \{u\}$ is not a maximal clique. Then $u$ and $u^q$ are Galois conjugates. We claim that $N(u)\neq N(u^q)$. If $x \in \F_q$, then $u^q-x=(u-x)^q$, thus $u-x \in H$ if and only if $u^q-x \in H$, and $u^q-x \in g^3H$ if and only if $u-x \in gH$. Thus, if $N(u)=N(u^q)$, then we must have $N(u)=N_1(u)$, which is impossible since $|N(u)|>|N_1(u)|$. Since $N(u)\cup \{u\}$ is not maximal, there is $v \in \F_{q^2} \setminus \F_q$ such that $v \neq u^q$, and $N(u) \cap N(v)=N(u)$. 

Next we estimate $|N(u) \cap N(v)|$ using Theorem~\ref{thm:main3}. As a preparation, we need to express the indicator function on $S$ using characters. Let $\chi$ be the multiplicative character of $\F_{q^2}$ with order $4$ such that $\chi(g)=i$. It follows that
$$\big|\big(\chi(x)+1\big)\big(\chi(x)+i\big)\big|^2
=
\begin{cases}
8, & \text{if } x \in H \cup gH,\\
0,  & \text{if } x \in g^2H \cup g^3H.
\end{cases}
$$
On the other hand, for $x \in \F_{q^2}$, we have
\begin{align*}
\big|\big(\chi(x)+1\big)\big(\chi(x)+i\big)\big|^2
&=\big(2+\chi(x)+\overline{\chi}(x)\big)\big(2+i\overline{\chi}(x)-i\chi(x)\big)\\
&=4+(2i+2)\overline{\chi}(x)+(2-2i)\chi(x)+i(\overline{\chi}^2(x)-\chi^2(x))\\
&=4+(2i+2)\overline{\chi}(x)+(2-2i)\chi(x).
\end{align*}
Therefore, $\mathbf{1}_S\colon \F_{q^2}^* \to \C$, the indicator function of~$S$, can be expressed as
$$
4\cdot\mathbf{1}_S=2+(i+1)\overline{\chi}+(1-i)\chi.
$$
It follows that
\begin{align*}
16|N&(u) \cap N(v)|\\
&=\sum_{a \in \F_q} 4\cdot\mathbf{1}_S(u-a) \cdot 4\cdot\mathbf{1}_S(v-a)\\
&=\sum_{a \in \F_q} \big(2+(1+i)\overline{\chi}(u-a)+(1-i)\chi(u-a)\big) \big(2+(1+i)\overline{\chi}(v-a)+(1-i)\chi(v-a)\big)\\
&=4q + 2(1+i)\sum_{a \in \F_q} \overline\chi(v-a) + 2(1-i) \sum_{a \in \F_q} \chi(v-a) \\
&\qquad{}+ 2(1+i) \sum_{a \in \F_q} \overline\chi(u-a) + 2i \sum_{a \in \F_q} \overline\chi(u-a)\overline\chi(v-a) + 2 \sum_{a \in \F_q} \overline\chi(u-a) \chi(v-a) \\
&\qquad{}+ 2(1-i) \sum_{a \in \F_q} \chi(u-a) + 2 \sum_{a \in \F_q} \chi(u-a)\overline\chi(v-a) - 2i \sum_{a \in \F_q} \chi(u-a)\chi(v-a).
\end{align*}
Since $u,v\in \F_{q^2} \setminus \F_q$ are not Galois conjugates, we may apply Theorem~\ref{thm:main3} to bound these eight character sums: the four sums with one character value can be bounded by~$\sqrt{q}$, while the four sums with two character values can be bounded by~$3\sqrt{q}$. Since $|2(1\pm i)|=2\sqrt2$, we obtain
$$
16 |N(u) \cap N(v)|\leq 4q + 4(2\sqrt2\cdot\sqrt q) + 4(2\cdot 3\sqrt q) = 4q + (8\sqrt2+24)\sqrt q.
$$
Since $q \geq 83$, we conclude that
$$
|N(u) \cap N(v)|\leq \frac{q}{4}+\frac{\sqrt{2}+3}{2}\sqrt{q}< \frac{q-1}{2}=|N(u)|,
$$
a contradiction. Hence, $N(u)\cup \{u\}$ is a maximal clique.
\end{proof}

\section{Remarks and open questions}\label{sec6}

We conclude this paper with some remarks and open questions.

\begin{rem}
Peisert \cite[Section 6]{P01} showed that the Paley graph $P_{q^2}$ and the Peisert graph $P^*_{q^2}$ of the same order are not isomorphic when $q \geq 7$ by showing that $\operatorname{Aut}(P_{q^2}) \neq \operatorname{Aut}(P^*_{q^2})$. While the fact that these two graphs are not isomorphic seems obvious (at least when $q$ is sufficiently large), to the best knowledge of the authors, there is no known simple proof. Peisert's original proof relies on heavy machinery from group theory. 

Here we sketch a new proof that~$P_{q^2}$ and~$P_{q^2}^*$ are not isomorphic when $q \equiv 3 \pmod 4$ and $q\geq 7$, using the structure of maximal cliques from $(\F_q,\alpha)$-constructions.
Suppose otherwise that these two graphs are isomorphic. Then there is a graph isomorphism $\phi\colon P_{q^2}^* \to P_{q^2}$ that maps the subfield~$\F_q$ to~$\F_q$. Indeed, any graph isomorphism must map~$\F_q$, a maximum clique in~$P_{q^2}^*$, to some maximum clique in~~$P_{q^2}$; while any maximum clique in~$P_{q^2}$ can be mapped to a maximum clique containing~$0$ and~$1$ via an automorphism of~$P_{q^2}$. But it is known~\cite{B84} that the only maximum clique in~$P_{q^2}$ that contains~$0$ and~$1$ is~$\F_q$ itself.

Choose any $\alpha^* \in \F_{q^2} \setminus \F_q$, and let~$C^*$ be the clique in~$P_{q^2}^*$ containing~$\alpha^*$ and its neighbors in~$\F_q$. Note that~$C^*$ is maximal in~$P_{q^2}^*$ by Theorem~\ref{thm:P*}. Similarly, set $\alpha=\phi(\alpha^*)$ and let~$C$ be the clique in $P_{q^2}$ containing~$\alpha$ and its neighbors in~$\F_q$. However,~$C=\phi(C^*)$ is not maximal in~$P_{q^2}$ since $C \cup \{\alpha^q\}$ is a maximal clique (as mentioned in the introduction), which contradicts the existence of the graph isomorphism~$\phi$. 
\end{rem} 

\begin{rem}\label{rem:more}
By combining the ideas in the proofs of Theorems~\ref{thm:newmaxclique} and~Theorem~\ref{thm:GP}, we can show the following: if $d \geq 2$ is fixed and $k$ is a non-negative integer, then there is a maximal clique of size approximately $q/d^k$ in $GP(q^2,d)$, provided that $q \equiv -1 \pmod d$ is sufficiently large. We can also prove the following by modifying the proof of Theorem~\ref{thm:P*}: if $k$ is a non-negative integer, then there is a maximal clique of size approximately $q/2^k$ in $P_{q^2}^*$, provided that $q \equiv 3 \pmod 4$ is sufficiently large.
\end{rem}

\begin{rem}
Theorem~\ref{thm:newmaxclique} says that (for suitable values of~$d$, $q$, and~$m$) there are maximal cliques in $GP(q^d,d)$ that are approximately $\frac1m$ as large as the clique~$\F_q$ itself, which is believed to be a maximum clique (essentially the only one, see \cite{Y23++} for a related discussion). These cliques have a specific structure, where we choose a small number of points outside of $\F_q$ and then add their common neighbors within $\F_q$. We believe that this construction yields all large maximal cliques in $GP(q^d,d)$ up to automorphisms of the graph and certain induced subgraphs, and we formulate a conjecture intended to serve as a sort of converse to Theorem~\ref{thm:newmaxclique}.

\begin{conj} \label{spectrum conj}
Fix an integer $d\ge2$. Consider all sequences $\{r_q / q\}$, indexed by odd prime powers $q\equiv1\pmod d$, such that $r_q$ is the cardinality of some maximal clique in the $d$-Paley graph $GP(q^d,d)$. Then the set of limit points of all such sequences equals $\{0\} \cup \{\frac1m\colon \rad(m) \mid \rad(d)\}$.
\end{conj}

\noindent Theorem~\ref{thm:newmaxclique} implies that every element of $\{0\} \cup \{\frac1m\colon \rad(m) \mid \rad(d)\}$ is such a limit point; we conjecture that there are no others.

As an example, consider all sequences $\{r_q / q\}$ indexed by odd prime powers $q$ such that $r_q$ is the cardinality of some maximal clique in the Paley graph $GP(q^2,2)$. Theorem~\ref{thm:newmaxclique} implies that every element of the set
$
\mathcal{L}_2=\{0\} \cup \{1,\frac12,\frac14,\frac18,\dots\}
$
is a limit point of some such sequence, and we conjecture that every limit point of any such sequence is in~$\mathcal{L}_2$. 

Baker, Ebert, Hemmeter, and Woldar \cite{BEHW96} (see also \cite{GKSV18,GMS22})  conjectured that the $(\F_q,\alpha)$-construct\-ion gives a second largest maximal clique in the Paley graph $GP(q^2,2)$; equivalently, there is no maximal clique of size between $\frac{q+\delta_q}{2}+1$ and $q-1$, where $\delta_q=1$ if $q \equiv 1 \pmod 4$ and $\delta_q=3$ if $q \equiv 3 \pmod 4$. To the best of our knowledge, no progress has been made towards their conjecture. A consequence of their conjecture would be that there are no limit points of any sequence $\{r_q/q\}$ in the interval $(\frac12,1)$. Our conjecture can thus be viewed as a further extension of their conjecture.  It is likely that new insights towards this conjecture would also shed light on estimating the clique number of Paley graphs of non-square order, a major open problem in additive combinatorics and analytic number theory \cite{CL07}. We refer to \cite{Y23++} and the references therein for recent progress towards the latter problem.

We similarly conjecture the analogous converses of the two statements in Remark~\ref{rem:more}: the analogous set of limit points for $GP(q^2,d)$ should be exactly $\mathcal{L}_d = \{0\} \cup \{ 1, \frac1d, \frac1{d^2}, \dots\}$, while the analogous set of limit points for $P^*_{q^2}$ should be exactly~$\mathcal{L}_2$.
\end{rem}

\section*{Acknowledgments}

The authors thank the anonymous referees for their valuable comments and suggestions.
The authors thank Daqing Wan for the clarification of a result in~\cite{W97} and Sergey Goryainov for sharing the conjecture related to Peisert graphs. The authors also thank Shamil Asgarli for helpful comments on a preliminary version of the manuscript. The first author was supported in part by a Natural Sciences and Engineering Council of Canada Discovery Grant. The research of the second author was supported in part by an NSERC fellowship.

\bibliographystyle{abbrv}
\bibliography{main}

\end{document}